\documentclass[a4paper,11pt]{article}

\usepackage[utf8]{inputenc}

\usepackage{amsmath}
\usepackage{amssymb}
\usepackage{amsthm}
\usepackage{graphicx}
\usepackage{caption}
\usepackage{fourier} 
\usepackage{soul} 

\usepackage{color}
\usepackage{ulem}

\usepackage[nottoc]{tocbibind} 

\usepackage[margin=3cm]{geometry}

\usepackage[T1]{fontenc} 

\usepackage{nicefrac}

\usepackage{cancel} 
\usepackage{hyperref}

\usepackage[textwidth=2.5cm]{todonotes}
\presetkeys{todonotes}{size=\scriptsize}{}

\newcommand\momentexpo{{\varsigma}}

\theoremstyle{plain}
\newtheorem{theorem}{Theorem}[section]
\newtheorem{proposition}[theorem]{Proposition}

\newtheorem{lemma}[theorem]{Lemma}

\theoremstyle{definition}

\theoremstyle{remark}
\newtheorem{remark}{Remark}[section]

\def\RR{\mathbb{R}}
\def\NN{\mathbb{N}}

\def\EE{\mathbb{E}}

\def\P{\mathbb{P}}
\def\R{\mathbb{R}}

\def\1{\mathbf{1}}

\def\Z{\mathbb{Z}}

\begin{document}

\title{Quasi-limiting behaviour of the sub-critical Multitype Bisexual Galton-Watson Branching Process}
\author{Coralie Fritsch$^{1}$, Denis Villemonais$^{1,2}$ and Nicolás Zalduendo$^{3}$}
\footnotetext[1]{Université de Lorraine, CNRS, Inria, IECL, F-54000 Nancy, France}
\footnotetext[2]{Institut universitaire de France (IUF)}
\footnotetext[3]{MISTEA, Université de Montpellier, INRAE, Institut Agro, Montpellier, France}

\maketitle

\begin{abstract}
    We investigate the quasi-limiting behavior of bisexual subcritical Galton-Watson branching processes. While classical subcritical Galton-Watson processes have been extensively analyzed,  bisexual  Galton-Watson branching processes present unique difficulties because of the lack of the branching property. To prove the existence of and convergence to one or several quasi-stationary distributions, we leverage on recent developments linking  bisexual  Galton-Watson branching processes extinction to the eigenvalue of a concave operator.
\end{abstract}

\section{Introduction}
\label{sec:introduction}

We study the quasi-limiting behaviour of bisexual subcritical Galton-Watson branching processes (bGWbp). This problem has been studied and largely solved decades ago for classical sub-critical Galton-Watson processes (see e.g.~\cite{Yaglom1947,SenetaVere-Jones1966,AthreyaNey1972}), but remains open for its bisexual counterpart. Recent developments in the theory of multi-dimensional bGWbp~\cite{FritschVillemonaisEtAl2022}  show that, under appropriate super-additivity assumptions, the extinction of these processes can be characterised by the eigenvalue $\lambda^*$ associated to a positively homogeneous concave operator and, more precisely, by its position  relatively to $1$: if $\lambda^*\leq 1$, then the process is eventually extinct almost surely, while, if $\lambda^*>1$, then, with positive probability, the process grows exponentially fast with rate equal or close to $\lambda^*$. We refer the reader to the surveys~\cite{Alsmeyer2002,Hull2003,Molina2010} on bGWbp and~\cite{Daley1968,DaleyHullEtAl1986,GonzalezHullEtAl2006,GonzalezMolina1996,zaldu2023} for further references. In this paper, we consider the sub-critical case $\lambda^*<1$ and show that the process admits infinitely many quasi-stationary distributions. Under an additional polynomial moment assumption, we prove that the process admits a finitely generated family of quasi-stationary distributions and, under an additional irreducibility assumption, a unique quasi-limiting distribution for compactly supported initial distributions. Although this is similar to the classical Galton-Watson case, $\lambda^*$ is not in general equal to the exponential survival rate of the process.

Let us introduce more formally our settings and assumptions. Consider $p,q \in \NN=\{0,1,2,\ldots\}$ with $p,q\geq 1$, a non-negative vector function $\xi: \NN^q\to \NN^p$ such that $\xi(0)=0$, and a family of integrable random vectors $V=(V_{i,\cdot})_{1\leq i \leq p}$, taking value in $\NN^q$ and whose expectation is denoted by~$\mathbb V$. We assume that $\sum_{i=1}^q \mathbb V_{i,j}>0$ for all $j\in\{1,\ldots,q\}$.  
We consider a process $Z=(Z_n)_{n\in\NN}$ on $\NN^p$ which represents the random sequence of the number of couples of each type in the population, and which evolves as follows: given $Z_{n-1}=(Z_{n-1,1},\ldots,Z_{n-1,p})$, we define for $n\geq 1$ the vector of children of the $n-$th generation $(W_{n,1},\ldots,W_{n,q})$ by
\[
W_{n,j}=\sum_{i=1}^p\sum_{k=1}^{Z_{n-1,i}} V_{i,j}^{(k,n)},\text{ for }1\leq j\leq q,
\]
where  $(V^{(k,n)})_{k,n\in\NN}$ is a family of i.i.d. copies of $V$.
Then, we set the vector of couples in the $n-$th generation as 
\[
Z_n=(Z_{n,1},\dots,Z_{n,p})=\xi(W_{n,1},\dots,W_{n,q}).
\]
The function $\xi$ is referred to as the \textit{mating function}. For instance, if $p=1$, $q=2$ and if $W_{n,1}$ represents the number of females  and $W_{n,2}$ the number of males in the $n-$th generation, classical choices for $\xi$ are $\xi(x,y)=\min\{x,y\}$ (the \textit{perfect fidelity} mating function) and $\xi(x,y) =x\min\{1,y\}$ (the \textit{promiscuous} mating function). If $p=q=2$, then $\xi(x,y)=(x,y)$  corresponds to the classical bi-dimensional Galton-Watson process.

For any probability measure $\mu$ on $\mathbb N^p$, we use the notation $\mathbb P_\mu$ and $\mathbb E_\mu$ for the law and associated expectation of the bisexual Galton-Watson process with initial number of couples $Z_0$ distributed according to $\mu$. As usual, we use the notations $\mathbb P_x$ and $\mathbb E_x$ when $\mu=\delta_x$ for some $x\in \mathbb N^p$. Note that $\xi(0)=0$ entails that $0$ is absorbing for $Z$.



We are interested in conditions ensuring the existence of a quasi-statio\-nary distribution for $Z$. We recall that a \textit{quasi-stationary distribution} for $Z$ is a probability measure $\nu_{QS}$ on $\mathbb N^p\setminus\{0\}$ such that, for all $n\geq 0$, $\mathbb P_{\nu_{QS}}(Z_n\neq 0)>0$ and
\begin{align*}
\mathbb P_{\nu_{QS}}(Z_n\in\cdot\mid Z_n\neq 0)=\nu_{QS}(\cdot).
\end{align*}
It is known that (we refere the reader to the book~\cite{ColletMartinezEtAl2013} and to the surveys~\cite{DoornPollett2013,MeleardVillemonais2012} for general properties and further examples), for any quasi-stationary distribution $\nu_{QS}$, there exists $\theta_{QS}\in(0,1]$ such that
\begin{align*}
\mathbb P_{\nu_{QS}}(Z_n\neq 0)=\theta_{QS}^n\text{ and }\mathbb P_{\nu_{QS}}(Z_n\in\cdot,\ Z_n\neq 0)=\theta_{QS}^n\nu_{QS}(\cdot).
\end{align*}
In what follows, $\theta_{QS}$ is referred to as the \textit{absorption parameter} of $\nu_{QS}$.

Throughout the paper, we make the assumption that $\xi$ is super-additive and sub-affine:

\medskip \noindent \textbf{Assumption~(S).} There exists $\alpha,\beta\in\mathbb R_+^p$ such that
\begin{equation}
\label{eq:superadditive}
 \alpha |x_1+x_2| +\beta \geq \xi(x_1+x_2)\geq \xi(x_1)+\xi(x_2),\ \forall x_1,x_2\in \NN^q,
\end{equation}
where $|\cdot|$ is the $\ell^1$-norm. 

\medskip
This assumption implies that the function $\mathfrak M:\mathbb R_+^p\mapsto \mathbb R_+^p$ given by
\[
\mathfrak M(z)= \lim_{k\to \infty} \dfrac{\xi (\lfloor kz\mathbb{V}\rfloor)}{k}
\]
is well defined,  bounded over $\mathbb S:=\left\{z\in \mathbb R_+^p,\ |z|=1\right\}$, positively homogeneous (i.e. $\mathfrak M(az)=a\mathfrak M(z)$ for all $a>0$ and all $z\in\mathbb R_+^p$) and concave. We refer the reader to Section~3 in~\cite{FritschVillemonaisEtAl2022}, where other properties of this functional are derived.  In addition, we assume that 

\medskip \noindent \textbf{Assumption~(P).}  $\mathfrak{M}$ is primitive, which means that there exists $n_0\geq 1$ such that $\mathfrak M^m(z) >0$ for all $m\geq n_0$ and $z\in \RR_+^p \setminus \{0\}$, where $\mathfrak M^m$ is the $m-$th iterate of $\mathfrak M$  and $\mathfrak M^m(z) >0$ means that all the coordinates of $\mathfrak M^m(z)$ are strictly positive.

\medskip In particular, this entails that there exists $\lambda^*> 0$ and $z^*\in\mathbb S^*:=\left\{z\in (\mathbb R_+\setminus\{0\})^p,\ |z|=1\right\}$ such that the limit 
\begin{align}
\label{eq:defP}
\lim_{n\to\infty} \dfrac{\mathfrak M^n (z)}{(\lambda^*)^n}=\mathcal P(z) z^*
\end{align}
exists, where $\mathcal P:\mathbb R_+^p\to \mathbb R_+$ is strictly positive on $\mathbb R_+^p\setminus\{0\}$, positively homogeneous and concave (see~\cite{Krause1994} for a general development on this theory  and Section~4.1 of~\cite{FritschVillemonaisEtAl2022} for an application to the context of bGWbp). In addition, according to Lemma~37 in~\cite{FritschVillemonaisEtAl2022}, $\left((\lambda^*)^{-n}\mathcal P(Z_n)\right)_{n\in\mathbb N}$ is a supermartingale. In particular, $\lambda^*<1$ implies that $(Z_n)_{n\in\mathbb N}$ goes extinct almost surely for any initial distribution (see~\cite{FritschVillemonaisEtAl2022} for details).


Our first result states that, in the sub-critical case, the process admits infinitely many quasi-stationary distributions,
indexed by $\theta\in [\upsilon_0,1)$, where
\begin{align}
\label{eq:upsilon0-def}
\upsilon_0:= \inf\{\upsilon>0,\text{ such that }\mathbb E_z(\upsilon^{-T_0})<+\infty,\ \forall z\in\mathbb N^p\}
\end{align}
with $T_0$ the extinction time.

\newcommand{\utheta}{\theta_0}
We also define the \textit{exponential convergence parameter} $\utheta$ by
\begin{align}
\label{eq:utheta_def}
\utheta = \sup_{z\in \mathbb N^p\setminus\{0\}}\sup\{\theta>0,\ \liminf_{n\to+\infty} \theta^{-n}\mathbb P_z(Z_n\neq 0)>0\},
\end{align}
with $\sup \emptyset=-\infty$. 

\begin{proposition}
\label{prop:theta0_leq_v0_lambda}
    If Assumptions~(S) and (P) hold true, then we have  $\utheta\leq \upsilon_0\leq \lambda^*$.
\end{proposition}



\begin{theorem}
    \label{thm:infinitely-many-QSDs} If Assumptions~(S) and (P) hold true and $\lambda^*<1$, then $Z$ admits a continuum of linearly independent quasi-stationary distributions:  for any $\theta\in[\upsilon_0,1)$, there exists a quasi-stationary distribution with absorption parameter $\theta$.
\end{theorem}

The proofs of Proposition~\ref{prop:theta0_leq_v0_lambda} and Theorem~\ref{thm:infinitely-many-QSDs} are detailed in Section~\ref{sec:infinity}, where we first state and prove a general result ensuring the existence of infinitely many quasi-stationary distributions for Feller processes, and then apply it to the sub-critical bGWbp.  We discovered during the finalization of this manuscript that the general result and its proof, which was part of the PhD thesis of one of the author~\cite{zaldu2023}, has also been proved independently for discrete state space processes in the recent preprint~\cite{benari2024}, using very similar methods.
We emphasize that in the classical Galton-Watson case, much more precise results are known (we recommend to the interested reader the paper~\cite{Maillard2018}, where the author proves a complete result for Galton-Watson processes and gives, in Section~3, a detailed exposition of the history of the problem, with its link to the identification of the Martin boundary of the classical Galton-Watson branching process). We also refer the reader to~\cite{FerrariKestenEtAl1995}, where a general existence result for quasi-stationary distributions is obtained by a renewal/compactness argument. 

\medskip
We consider now the problem of existence of a finitely generated set of quasi-stationary distributions satisfying an integrability assumption and of convergence of conditional laws toward a quasi-stationary distribution. Since $\mathfrak M$ is concave on $\mathbb R_+^p$, it is locally Lipschitz on $\mathbb S^*$. We add the following regularity assumption on $\mathfrak M$, where  we use the notation, for all $I\subset \{1,\ldots,p\}$,
\[
\mathcal Z_I:=\{x\in\mathbb S\text{ such that }\ x_i>0\ \forall i\in I,\ x_i=0\ \forall i\notin I\},
\]
where, here and across the paper, $x_i$ denotes the $i^{\text{th}}$ coordinate of $x$.

\medskip\noindent \textbf{Assumption~(C).}
For all $I\neq \emptyset$, $\mathfrak M$ is uniformly continuous over $\mathcal Z_I$.

\medskip
We also consider the  following moment condition, related to the exponential convergence parameter $\utheta$.



\medskip\noindent\textbf{Assumption~(M).}
There exists $\momentexpo>1$ such that $(\lambda^*)^\momentexpo<\utheta$ and such that $\EE(V_{i,j}^\momentexpo) <+\infty$ for all $i,j$. 

\medskip We emphasize that Assumption~(M) implies $\lambda^*<1$,  so that the following results apply in the setting of the sub-critical bGWbp. This assumption also requires implicitly that $\theta_0>0$, which is the case in many situations.

\medskip

The proof of the following theorem is detailed in Section~\ref{sec:proof-of-thm:finitely-many-QSDs}. It is based on~\cite{ChampagnatVillemonais2022} and makes use of the following function and measure spaces. Given a positive function $\varphi$ on $\mathbb N^p\setminus\{0\}$, we set
\[
L^\infty(\varphi):=\{f:\mathbb N^p\setminus\{0\}\to\mathbb R,\ \|f/\varphi\|_\infty<+\infty\},\text{ endowed with the norm }\|f\|_\varphi=\|f/\varphi\|_\infty
\]
and
\[
\mathcal M(\varphi):=\{\mu\text{ non-negative measure on }\mathbb N^p\setminus\{0\}\text{ such that }\mu(\varphi)<+\infty\}
\]
where $\mu(\varphi)=\int_{\mathbb N^p\setminus\{0\}}\varphi(z)\mu(\mathrm{d} z)$.

Finally, when Assumption~(M) is enforced, we define, for all $a\in(1,\momentexpo)$  such that $(\lambda^*)^a<\theta_0$,
the function $Q_a:\NN^p\setminus\{0\}\longrightarrow \RR_+$  by 
\[
Q_a(z)=\mathcal P(z)^a.
\]
Note that $\inf Q_a>0$.

In the following theorem, we say that a process $Z$ is aperiodic if, for all $z\in \mathbb N^p$, either $\mathbb P_z(\exists n\geq 1, Z_n=z)=0$ or there exists $n_0(z)\geq 0$ such that, $\forall n\geq n_0(z)$, $\mathbb P_z(Z_n=z)>0$.

\begin{theorem}
    \label{thm:finitely-many-QSDs} We assume that Assumptions~(S), (P), (C) and (M) are in place and that the process $Z$ is aperiodic. We fix $a\in(1,\momentexpo)$  such that $(\lambda^*)^a<\theta_0$.
    
    Then
    there exists $\ell\geq 1$ and a family of quasi-stationary distributions $\nu_1,\ldots,\nu_\ell\in \mathcal
    M(Q_a)$ with absorption parameter $\theta_0$ for $Z$ such that any quasi-stationary distribution $\nu\in\mathcal M(Q_a)$ with absorption parameter $\theta_0$  for $Z$ is a convex combination of $\nu_1,\ldots,\nu_\ell$. In addition, there exists    
    a bounded function $j:\mathbb N^p\setminus\{0\}\to \mathbb N$ 
    and there exists, for each $i\in \{1,\ldots,\ell\}$, a non-negative non-identically zero function $\eta_{i}\in L^\infty(Q_a)$   such that, for all $f\in L^\infty(Q_a)$, all $n\geq 1$ and all $z\in
    \mathbb N^p\setminus\{0\}$,
    \begin{equation}
    \label{eq:eta}
    \left|\theta_{0}^{-n} n^{-j(z)} \EE_z(f(Z_n)\mathbf 1_{Z_n\neq 0})-\sum_{i=1}^\ell \eta_{i}(z)\nu_{i}(f)\right|\leq \alpha_{n} Q_a(z)\|f\|_{Q_a},
    \end{equation}
    where $\alpha_{n}$ goes to $0$ when $n\to+\infty$ which does not depend on $f$ nor on $z$ (but may depend on $a$).

   In addition, the set $E:=\{z\in \mathbb N^p\setminus\{0\},\ \sum_{i=1}^\ell\eta_i(z)=0\}$ is finite, and any quasi-stationary distributions for $Z$ in $\mathcal
    M(Q_a)$ with absorption parameter strictly smaller than $\theta_0$ is supported by~$E$.

    Finally, there is no quasi-stationary distribution for $Z$ in $\mathcal
    M(Q_a)$ with absorption parameter strictly larger than $\theta_0$ and, for all $z\in \mathbb N^p\setminus\{0\}$, the conditional distribution $\mathbb P_z(Z_n\in\cdot\mid Z_n\neq 0)$ converges in $\mathcal M(Q_a)$.
    
\end{theorem}


\begin{remark}
    Note that, under the assumptions of 
Theorem~\ref{thm:finitely-many-QSDs}, we have $\upsilon_0=\utheta$.  
In fact, taking $f\equiv 1$ in~\eqref{eq:eta} leads to $\mathbb{P}_z(Z_n\neq 0)\sim_{n\to \infty} \theta_{0}^{n} n^{j(z)}\sum_{i=1}^\ell \eta_{i}(z)$. Then for all $v>\theta_0$, for all $z\in \mathbb N^p\setminus \{0\}$, we have 
$\mathbb E_z(v^{-T_0}) = \sum_{n\in\mathbb{N}} v^{-n}\left(\mathbb P_z(Z_{n-1}\neq 0)-\mathbb P_z(Z_{n}\neq 0)\right)<\infty$. Then $v_0 \leq \theta_0$, and Proposition~\ref{prop:theta0_leq_v0_lambda} leads to the equality. 
\end{remark}

\medskip Under the additional assumption that the process $Z$ is irreducible (meaning that $\forall x,y\in\mathbb N^p\setminus \{0\}$, there exists $n\geq 0$ such that $\mathbb P_x(Z_n=y)>0$), one obtains that, for all $z\in\mathbb N^p\setminus\{0\}$, 
\begin{align*}
\utheta = \sup\{\theta>0,\ \liminf_{n\to+\infty} \theta^{-n}\mathbb P_z(Z_n\neq 0)>0\},
\end{align*}
does not depend on $z$. In the irreducible case, we have the following result, proved in Section~\ref{sec:lastone}.

\begin{theorem}
    \label{thm:unique-QSD}
     Assume that Assumptions~(S), (P), (C) and (M) hold true and that the process $Z$ is aperiodic and irreducible.  We fix $a\in(1,\momentexpo)$  such that $(\lambda^*)^a<\theta_0$.

      Then the process $Z$ admits a unique quasi-stationary distribution $\nu_{QS}$ in $\mathcal M(Q_a)$. Its absorption parameter $\theta_{QS}$ is equal to $\theta_0$ and there exists a unique positive function $\eta_{QS}$ in $L^\infty(Q_a)$, such  that
    \begin{align}
    \label{eq:thmQSD2}
    \mathbb P_{\nu_{QS}}(Z_n\neq 0)=\theta_0^n \quad \text{ and } \quad \mathbb E_z(\eta_{QS}(Z_n)\mathbf{1}_{Z_n\neq 0})=\theta_0^n\eta_{QS}(z),\quad\forall n\geq 0,\,z\in \mathbb N^p\setminus\{0\}
    \end{align}
    and $\eta_{QS}$ is lower bounded away from $0$.
    In addition, there exists $C>0$ and $\gamma\in(0,1)$ such that, for any probability measure $\mu\in\mathcal M(Q_a)$ and any function $f\in L^\infty(Q_a)$, we have
    \begin{align}
    \label{eq:thmQSD1}
    \left|\theta_0^{-n}\mathbb E_\mu\left(f(Z_n)\mathbf{1}_{Z_n\neq 0}\right)-\mu(\eta_{QS})\nu_{QS}(f)\right|\leq C\gamma^n \mu(Q_a)\,\|f\|_{Q_a},\quad\forall n\geq 0
    \end{align}
    and
    \begin{align}
    \label{eq:thmQSD3}
    \left|\mathbb E_\mu\left(f(Z_n)\mid Z_n\neq 0\right)-\nu_{QS}(f)\right|\leq C\gamma^n\,\mu(Q_a)\,\|f\|_{Q_a},\quad\forall n\geq 0.
    \end{align}
    Finally, for all probability measure $\mu\in \mathcal M(\eta_{QS})$, 
    \begin{align}
    \label{eq:thmQSD4}
    \mathbb P_\mu\left(Z_n\in \cdot\mid Z_n\neq 0\right)\xrightarrow[n\to+\infty]{TV} \nu_{QS},
    \end{align}
where $TV$ refers to the total variation convergence.
\end{theorem}

Theorem~\ref{thm:unique-QSD} states the existence and the uniqueness of a quasi-stationary distribution $\nu_{QS}$ integrating the Lyapunov function $Q_a$.
Moreover, \eqref{eq:thmQSD3} and \eqref{eq:thmQSD4} establish the convergence of the law of the process conditioning to the non-extinction towards $\nu_{QS}$ for all initial distribution $\mu$ integrating $Q_a$ or the eigenfunction $\eta_{QS}$, with moreover an exponential speed of convergence for $\mu$ integrating $Q_a$. In particular, the convergence towards $\nu_{QS}$ holds for any Dirac mass $\mu=\delta_z$, with $z\in E$, meaning that $\nu_{QS}$ is a Yaglom limit.
In addition, \eqref{eq:thmQSD1} gives the convergence of the law of the process without conditioning. In particular, taking $f\equiv 1$ in \eqref{eq:thmQSD1} leads to the speed of convergence of the extinction probability towards 0 
at speed exactly $\theta_0$.

\medskip

The proofs of the last two results rely in part on recent advances in the theory of quasi-statio\-nary distributions, which allow to derive several properties from Foster-Lyapunov type criteria, given in our case by the following proposition. 


\begin{proposition}
    \label{prop:lyap}
    If Assumptions~(S), (P), (C) and (M) hold true, then, for any $a\in(1,\momentexpo)$ such that $(\lambda^*)^a<\theta_0$, we have
    \begin{equation}\label{ineq:QSD}
    \EE(Q_a(Z_1)\mathbf{1}_{Z_1\neq 0} |Z_0=z)\leq \theta_{a} Q_a(z)+C_a,\: \forall z\in\NN^p\setminus \{0\}
    \end{equation}
    for some constants $\theta_a\in [0,\theta_0)$ and $C_a>0$.
\end{proposition}

\begin{remark}
    We already observed that $\theta_0\leq \lambda^*$.  One also checks that, at least under the assumptions of Theorem~\ref{thm:unique-QSD}, 
    if $\mathcal P$ is not linear, then $\utheta<\lambda^*$. Indeed, under these assumptions, let $\nu_{QS}$ be the quasi-stationary distribution from Theorem~\ref{thm:unique-QSD}, so that
    \begin{align*}
        \theta_0^n \nu_{QS} (\mathcal P)=\sum_{z\in\mathbb N^p\setminus\{0\}} \nu_{QS}(\{z\}) \mathbb E_z\left(\mathcal P (Z_n)\right).
    \end{align*}
    Since under the assumptions of Theorem~\ref{thm:unique-QSD}, the process is irreducible, the support of $\nu_{QS}$ is $\mathbb N^p\setminus\{0\}$ and hence $\mathcal P$ is concave not affine on the support of $\nu_{QS}$. Using~\eqref{eq:thmQSD3}, we deduce that there exists $z_0\in \mathbb N^p\setminus\{0\}$ and $n\geq 1$ such that $\mathcal P$ is concave not affine on the support of $\mathbb P_{z_0}(Z_n\in\cdot)$, so that, by Jensen (strict) inequality, 
   \[
   \mathbb E_{z_0}\left(\mathcal P(Z_n)\right)< \mathcal P(\mathbb E_{z_0}(Z_n))\leq \mathcal P(\mathfrak M^n(z_0))=(\lambda^*)^n \mathcal P(z_0),
   \]
    where we used $\mathbb E_{z_0}(Z_n)\leq \mathfrak M^n(z_0)$ (see Lemma~27 in~\cite{FritschVillemonaisEtAl2022}) and the fact that $\mathcal P$ is superadditive (since $\mathfrak M$ is superadditive by Assumption~(S) and by definition of $\mathcal P$ given by~\eqref{eq:defP}) for the last inequality, and the definition of $\mathcal P$ (see \eqref{eq:defP}) for the last equality. Similarly, we have 
    $\mathbb E_{z}\left(\mathcal P(Z_n)\right)\leq (\lambda^*)^n \mathcal P(z)$ for all $z\in\mathbb N^p\setminus \{0\}$. We conclude that
    \begin{align*}
        \theta_0^n \nu_{QS} (\mathcal P) < \sum_{z\in\mathbb N^p\setminus\{0\}} \nu_{QS}(\{z\}) (\lambda^*)^n \mathcal P (z) =(\lambda^*)^n  \nu_{QS} (\mathcal P).
    \end{align*}
    This implies that $\theta_0< \lambda^*$ when $\mathcal P$ is not linear.

    It is also known that $\mathfrak M(z)=\sup_{k\geq 1} \frac1k \mathbb E_{kz}(Z_1)$ (see Corollary~7 in~\cite{FritschVillemonaisEtAl2022}), and the same calculation as above shows that, if $\mathbb E_{z_0}(Z_1) <\mathfrak M(z_0)$ for some $z_0\in\mathbb N^p$, then $\utheta<\lambda^*$ (even if $\mathcal P$ is linear).
\end{remark}

\begin{remark}    
    When $\utheta=\lambda^*$, 
    Assumption~(M)
    requires polynomial moments for $V$, with exponents that can be arbitrarily close to $1$. 
We leave as an open question whether, in this case, $L\log L$ type criteria can be obtained, as it is the case in the classical Galton-Watson case (for which $\utheta=\lambda^*$).

In the situation where $\utheta<\lambda^*$, one can not choose $\momentexpo$ arbitrarily close to $1$ in Assumption~(M), since it imposes $\momentexpo > \frac{|\log \utheta|}{|\log \lambda^*|}$. We leave as an open question whether $\EE(V_{i,j}^\momentexpo) <+\infty$ for some $\momentexpo \leq  \frac{|\log \utheta|}{|\log \lambda^*|}$ may be sufficient to obtain the conclusion of Theorem~\ref{thm:finitely-many-QSDs} or~\ref{thm:unique-QSD}, at least in some particular cases.  
\end{remark}

In Section~\ref{sec:infinity}, we state a general result for the existence of infinitely many quasi-stationary distributions, then we prove Proposition~\ref{prop:theta0_leq_v0_lambda} and Theorem~\ref{thm:infinitely-many-QSDs}. Section~\ref{sec:proof_prop_lyap} is devoted to the proof of Proposition~\ref{prop:lyap}.
Theorems~\ref{thm:finitely-many-QSDs} and \ref{thm:unique-QSD} are proved in Sections~\ref{sec:proof-of-thm:finitely-many-QSDs} and \ref{sec:lastone} respectively. For the proofs of the last two theorems, we make use of Assumption~(E) of~\cite{ChampagnatVillemonais2017} which is recalled in Appendix. In Appendix we also state a general result implying Assumption~(E) of~\cite{ChampagnatVillemonais2017}. 

\section{A simple criterion for the existence of infinitely many quasi-sta\-tionary distributions and application to the bGWbp}
\label{sec:infinity}


In Section~\ref{sec:abstract}, we first state and prove  a result on the existence of a continuum of quasi-stationary distributions for general sub-Markov kernels on Polish spaces. In Section~\ref{sec:infinity}, we apply this result to the bGWbp.

\subsection{General setting}
\label{sec:abstract}
Let $(X,d)$ be a Polish space endowed with a $\sigma$-algebra $\mathcal X$ that contains the toplogy induced by $d$.  Let $K$ be a sub-Markov kernel on $X$, which means that $K$ is a non-negative kernel on $X$ such that $K(x,X)\leq 1$ for all $x\in X$. We assume that $K$ has Feller regularity, meaning that the associated operator $K:f\mapsto Kf:=\int_X K(\cdot,\mathrm dy) f(y)$ is a well defined functional from $C_b(X)$ to $C_b(X)$ (the set of continuous bounded functions on $X$). We say that a probability measure $\mu$ on $X$ is a quasi-stationary distribution for $K$ if, for some $\lambda\in(0,1]$,
\begin{align*}
\int_X \mu(\mathrm dx) K(x,\cdot)=\lambda \mu(\cdot).
\end{align*}
Our aim is to prove that $K$ admits infinitely many quasi-stationary distributions under mild assumptions.

In the following theorem, we say that a sequence $(x_n)_{n\in\mathbb N}\in X^{\mathbb N}$ tends to infinity if it eventually leaves any compact set: for all compact subset $L$ of $X$, there exists $n_L\geq 0$ such that, for all $n\geq n_L$, $x_n\notin L$. Similarly, we say that a real functional on $X$ tends to infinity when $x\to+\infty$ if, for all $R>0$, there exists a compact subset $L_R$ of $X$ such that the functional is bounded below by $R$ on $X\setminus L_R$. We also define the iterated kernels $K^n$ as usual by
\begin{align*}
K^0(x,A)=\mathbf 1_{x\in A},\ K^{\ell+1}(x,A)=\int_{X} K^\ell(x,\mathrm dy) K(y,A)\ \forall x\in X, A\in \mathcal X, \ell\geq 0.
\end{align*}


\begin{theorem}
    \label{thm:infinity}
    Assume that there exists $a>0$ and  $(x_n)_{n\in\mathbb N}\in X^{\mathbb N}$ tending to infinity such that
    \begin{align}
    \label{eq:thm:infinity1}
    \sum_{\ell=0}^\infty a^{-\ell} K^\ell(x_n,X) <+\infty,\ \forall n\geq 0
    \end{align}
    and let
    \begin{align}
        a_0:=\inf \{a>0,\text{ s.t. there exists $(x_n)_{n\in\mathbb N}\in X^{\mathbb N}$ tending to infinity such that~\eqref{eq:thm:infinity1} holds true.}\}
    \end{align}
    Assume also that there exists $a_1>a_0$ such that
    \begin{align}
    \label{eq:thm:infinity2}
    \sum_{\ell=0}^\infty a_1^{-\ell} K^\ell(x,X)\xrightarrow[x\to+\infty]{}+\infty.
    \end{align}
    Then $K$ admits an infinite family of quasi-stationary distributions 
    $(\mu_\lambda)_{\lambda\in [a_0,a_1)}$, with $\int_X \mu_\lambda(\mathrm dx) K(x,\cdot)=\lambda \mu_\lambda$ for all $\lambda\in [a_0,a_1)$.
\end{theorem}

\begin{proof}[Proof of Theorem~\ref{thm:infinity}]
    Fix $\lambda\in (a_0,a_1)$. Since $\lambda> a_0$ then there exists $(x_n)_{n\in\mathbb N}\in X^{\mathbb N}$ tending to infinity such that, for all $n\geq 0$, the probability measure
    \begin{align*}
    \mu_{\lambda,n} =\frac{\sum_{\ell=0}^\infty \lambda^{-\ell} K^\ell(x_n,\cdot)}{\sum_{\ell=0}^\infty \lambda^{-\ell} K^\ell(x_n,X)}
    \end{align*}
    is well defined.
    We have, for all $n\geq 0$,
    \begin{align}
    \mu_{\lambda,n} K (\cdot)&:=\int_X \mu_{\lambda,n}(\mathrm dx) K(x,\cdot)=\frac{\sum_{\ell=0}^\infty \lambda^{-\ell} K^{\ell+1}(x_n,\cdot)}{\sum_{\ell=0}^\infty \lambda^{-\ell} K^\ell(x_n,X)}\notag \\
    &=\frac{\lambda\sum_{\ell=0}^\infty \lambda^{-(\ell+1)} K^{\ell+1}(x_n,\cdot)}{\sum_{\ell=0}^\infty \lambda^{-\ell} K^\ell(x_n,X)}\notag \\
    &=\frac{\lambda\sum_{\ell=0}^\infty \lambda^{-\ell} K^{\ell}(x_n,\cdot)-\lambda\delta_{x_n}(\cdot)}{\sum_{\ell=0}^\infty \lambda^{-\ell} K^\ell(x_n,X)}\label{eq:munK} \\
    &\leq \lambda \mu_{\lambda,n}(\cdot).\notag
    \end{align}
    We deduce that, for all $\ell\geq 1$,
    \begin{align*}
    \mu_{\lambda,n} K^\ell\leq \lambda^\ell \mu_{\lambda,n}.
    \end{align*}
    In particular, we obtain that 
    \begin{align}
    \label{eq:domination}
    \int_X \mu_{\lambda,n}(\mathrm dx)\sum_{\ell=0}^\infty a_1^{-\ell} K^\ell(x,X)=\sum_{\ell=0}^\infty a_1^{-\ell} \mu_{\lambda,n} K^\ell(X)\leq \sum_{\ell=0}^\infty \left(\frac{\lambda}{a_1}\right)^{\ell}=\frac{a_1}{a_1-\lambda}.
    \end{align}
    Since we assumed that $\sum_{\ell=0}^\infty a_1^{-\ell} K^\ell(x,X)$ goes to infinity when $x\to\infty$, this and Prokhorov's Theorem imply that the sequence of measures $(\mu_{\lambda,n})_{n\in\mathbb N}$ is relatively compact for the weak convergence topology.
    
    Let $\mu_\lambda$ be any adherent point of $(\mu_{\lambda,n})_{n\in\mathbb N}$, and observe that $\mu_\lambda$ is a probability measure. Since $K$ is assumed to preserve bounded continuous functions, we deduce that, up to a subsequence and for the weak convergence topology,
    \begin{align*}
    \mu_{\lambda} K = \lim_{n\to+\infty} \mu_{\lambda,n} K.
    \end{align*}
    But we computed (see~\eqref{eq:munK})
    \begin{align*}
    \mu_{\lambda,n} K= \lambda \mu_{\lambda,n} -\frac{\lambda\delta_{x_n}}{\sum_{\ell=0}^\infty \lambda^{-\ell} K^{\ell}(x_n,X)},
    \end{align*}
    where, by assumption, 
    \begin{align*}
    \sum_{\ell=0}^\infty \lambda^{-\ell} K^{\ell}(x_n,X)\geq \sum_{\ell=0}^\infty a_1^{-\ell} K^{\ell}(x_n,X)\xrightarrow[n\to+\infty]{} +\infty.
    \end{align*}
    We deduce that $\lim_{n\to+\infty} \mu_{\lambda,n} K=\lim_{n\to+\infty}\lambda \mu_{\lambda,n}$, which is equal to $\lambda \mu_{\lambda}$. We thus proved that $\mu_{\lambda} K=\lambda \mu_{\lambda}$, which means that $\mu_{\lambda}$ is a quasi-stationary distribution for $K$ with absorption parameter $\lambda$.

    It remains to prove that there exists a quasi-stationary distribution with absorption parameter $a_0$. Since the set of probability measures $\mu_\lambda$, $\lambda\in (a_0,(a_0+a_1)/2)$, satisfies~\eqref{eq:domination} with a uniform upper bound, we deduce that it is relatively compact. One  checks as above that any limit point, when $\lambda\to a_0$, of this family is a quasi-stationary distribution for $K$ with absorption parameter $a_0$. This concludes the proof of Theorem~\ref{thm:infinity}.
\end{proof}

\begin{remark}
    One easily checks from the proof that the condition that $K$ preserves bounded continuous functions is  stronger than necessary. Actually, it would be sufficient to assume that $K$ sends a generating family of bounded continuous functions into a family of bounded continuous functions. For instance, it sufficient to assume that $K$ sends compactly supported continuous functions to bounded continuous functions.
\end{remark}

\begin{remark}
    One can easily extends this result to the general setting of non-conservative kernels $K$. More precisely, assume that there exists a positive function $h$ on $X$ such that $Kh\leq c h$ for some $c>0$. Then the kernel defined by
    \begin{align*}
    \widetilde K(x,A)=\frac{1}{ c h(x)}\,\int_A K(x,\mathrm dy) h(y),\ \forall x\in X,\ \forall A\in\mathcal X,
    \end{align*}
    is a sub-Markov kernel to which Theorem~\ref{thm:infinity} may apply. Note that the assumptions would translate into: for any continuous function $f:X\to \mathbb R$, $x\in X\mapsto \frac{1}{ h(x)}\int_X K(x,\mathrm dy) h(y) f(y)$ is continuous; for all $n\geq 0$, $\sum_{\ell=0}^\infty a_0^{-\ell} \int_X K^\ell(x_n,\mathrm dy) h(y)<+\infty$ and $\frac{1}{h(x)} \sum_{\ell=0}^\infty a_1^{-\ell} \int_X K^\ell(x,\mathrm dy) h(y)\xrightarrow[x\to+\infty]{}+\infty$. In addition, the resulting quasi-stationary measures may not be finite measures if $h$ is not lower bounded (we only obtain that $\int_X \mu_\lambda(\mathrm dy)h(y)<+\infty$ for all $\lambda\in[a_0,a_1)$). On the other hand, if $h$ is lower bounded away from $0$, then $\mu_\lambda$ is a finite measure. 
\end{remark}

\subsection{Proof of Proposition~\ref{prop:theta0_leq_v0_lambda} and Theorem~\ref{thm:infinitely-many-QSDs}}
\label{sec:proof_thm_infinitely-many-QSDs}
In the context of bGWbp, we choose $K(x,\mathrm dy)=\mathbb P_x(Z_1\in\mathrm dy,\,Z_1\neq 0)$ for all $x\in X=\mathbb N^p\setminus\{0\}$. We first prove Proposition~\ref{prop:theta0_leq_v0_lambda}, as well as that $a_0\leq \upsilon_0$, where $a_0$ is defined in Theorem~\ref{thm:infinity} and $\upsilon_0$ is defined in~\eqref{eq:upsilon0-def}. Then we prove Theorem~\ref{thm:infinitely-many-QSDs}.

First, from their definitions, it is clear that $\upsilon_0\leq 1$ and $\theta_0\leq 1$. For all $\theta<\theta_0$, there exists $v\in(\theta,1)$ and $z\in \mathbb{N}^p\setminus\{0\}$ such that $\liminf_{n\to \infty} v^{-n}\mathbb{P}_z(Z_n\neq 0)>0$.  For such $v$ and $z$, we have $\mathbb{E}_z(v^{-T_0})= \sum_{\ell=0}^\infty v^{-\ell} \mathbb{P}_z\left(T_0=\ell\right)$ with
\begin{align*}
\liminf_{\ell \to \infty}v^{-\ell} \mathbb{P}_z\left(T_0=\ell\right)
	&= \liminf_{\ell \to \infty} v^{-\ell} \left(\mathbb{P}_z(Z_{\ell-1} \neq 0) -\mathbb{P}_z(Z_{\ell} \neq 0)\right)\\
	&=   (v^{-1}-1)\liminf_{\ell \to \infty}v^{-\ell}\mathbb{P}_z(Z_{\ell} \neq 0)
	>0.
\end{align*}
Then $\mathbb{E}_z(v^{-T_0})=\infty$ and hence $v\geq \upsilon_0$. We deduce that $\upsilon_0\geq\theta$ and hence $\upsilon_0\geq \theta_0$.


In order to prove that $a_0\leq\upsilon_0$, we first observe that, for  all $a>\upsilon_0$, there exists $v\leq 1$, such that $v<a$ and for all $z\in\mathbb N^p\setminus\{0\}$, $\mathbb E_z(v^{-T_0})<\infty$. For such $v$ we then obtain
\begin{align*}
\sum_{\ell\geq 0}a^{-\ell}K^{\ell}(z,X)=\sum_{\ell\geq 0}a^{-\ell}\mathbb P_z(T_0\geq \ell+1)\leq \sum_{\ell\geq 0}a^{-\ell}v^{\ell+1}\mathbb E_z(v^{-T_0})\leq C_{a,v} \mathbb E_z(v^{-T_0})<\infty.
\end{align*}
This shows that $a_0\leq \upsilon_0$.


Using the fact that $\mathcal P$ is lower bounded away from $0$ on $\mathbb N^p\setminus\{0\}$ 
(as $\mathcal P$ is strictly positive and positively homogeneous) and that $K(\cdot,X)$ is bounded by $1$ and vanishes at $0$
, we deduce that there exists a constant 
$c>0$ such that, for all $\ell\geq 1$ and all $x\in\mathbb N^p$,
\begin{align*}
K^\ell(x,X)\leq \frac1c \mathbb E_x\left(\mathcal P(Z_\ell)\right).
\end{align*}
In addition, $(\lambda^*)^{-n} \mathcal P(Z_n)$ is a supermartingale (see Lemma~37 in~\cite{FritschVillemonaisEtAl2022}), so that, for any $a\in (\lambda^*,1)$,
\begin{align*}
\sum_{\ell=0}^\infty a^{-\ell} K^\ell(x,X) \leq \frac1c\sum_{\ell=0}^\infty a^{-\ell} (\lambda^*)^\ell \mathcal P(x) <+\infty.
\end{align*}
This shows that $\upsilon_0\leq \lambda^*$, then Proposition~\ref{prop:theta0_leq_v0_lambda} holds.

Since $\mathfrak M$ is primitive by Assumption~(P), Theorem~6 in~\cite{FritschVillemonaisEtAl2022} implies that, for all $z\in \mathbb N^p\setminus\{0\}$ and for all $n\geq n_0$, there exists $c_0>0$ and $k_0\geq 1$ such that for all $k\geq k_0$
\begin{align*}
\mathbb P(Z_{n}\geq c_0 (1,\ldots,1)\mid Z_0=k z)>0.
\end{align*}
Applying this result for $z$ taken in each element of the canonical basis of $\mathbb N^p$ and using the super-additivity of $\xi$, one deduces that
, for any $n\geq n_0$, there exist constants $r>0$ and $c'_0>0$ such that
\begin{align*}
\inf_{|z|>r} \mathbb P(Z_{n}\geq c'_0 (1,\ldots,1)\mid Z_0= z)>0.
\end{align*}
 Using again the superadditivity of the model, we obtain that
 , for all $z_1,\ldots,z_\ell$, with $\ell\geq 1$ and $|z_i|>r$,
\begin{align*}
\mathbb P_{z_1+\cdots+z_\ell}(Z_{n}=0)\leq \mathbb P_{z_1}(Z_n=0)\cdots \mathbb P_{z_\ell}(Z_n=0)\leq \big(1- \inf_{|z|>r} \mathbb P(Z_{n}\geq c'_0 (1,\ldots,1)\mid Z_0= z)\big)^\ell.
\end{align*}
and hence
\begin{align*}
\mathbb P_{z}(Z_{n}=0)\xrightarrow[|z|\to+\infty]{} 0.
\end{align*}
In particular, we deduce from Fatou's Lemma that, for any $a_1\in(0,1)$, 
\begin{align*}
\sum_{\ell=0}^\infty a_1^{-\ell} \mathbb P_z(Z_\ell\neq 0)\xrightarrow[|z|\to+\infty]{}+\infty.
\end{align*}
and hence~\eqref{eq:thm:infinity2} holds true for any $a_1\in(a_0,1)\supset (\upsilon_0,1)$.

We deduce that Theorem~\ref{thm:infinity} applies to the subcritical bGWbp. Since quasi-stationary distributions with different absorption parameters are linearly independent, this proves Theorem~\ref{thm:infinitely-many-QSDs}.

\section{Proof of Proposition~\ref{prop:lyap}}
\label{sec:proof_prop_lyap}

We start with the following technical lemma on the regularity of $\mathcal P$.
\begin{lemma}
    \label{lem:useful2}
Assume that Assumptions~(S), (P), (C) and (M) hold true. Then, for all $I\subset \{1,\ldots,p\}$, $I\neq \emptyset$, the operator $\mathcal P$ restricted to $\mathcal Z_I$ is uniformly continuous.
\end{lemma}

\begin{proof}[Proof of Lemma~\ref{lem:useful2}]
    Since $\mathcal P \circ \mathfrak M^{n_0}=\left(\lambda^*\right)^{n_0} \mathcal P$, it is sufficient to prove that $\mathcal P \circ \mathfrak M^{n_0}$ is uniformly continuous on $\mathcal Z_I$. But $\mathfrak M^{n_0} (\mathcal Z_I)\subset \mathfrak M^{n_0}(  \mathbb S)$ which is relatively compact in $(\mathbb R_+\setminus\{0\})^p$ (since $\mathfrak M$ is primitive and bounded by assumption), and $\mathcal P$ is concave and hence it is locally Lipschitz in $(\mathbb R_+\setminus\{0\})^p$. It is thus sufficient to prove that $\mathfrak M^{n_0}$ is uniformly continuous on $\mathcal Z_I$. 

    Let us first observe that $\mathfrak M$ is locally uniformly continuous on $(\varepsilon+\mathbb R_+\setminus\{0\})\,\mathcal Z_I$ for any $\varepsilon>0$, since $\mathfrak M$ is uniformly continuous on $\mathcal Z_I$,  bounded on $\mathbb S\supset \mathcal Z_I$, and positively homogeneous. In order to conclude, it remains to prove that
    there exists  $J\subset\{1,\ldots,p\}$ and $\varepsilon>0$ such that $\mathfrak M(\mathcal Z_I)\subset (\varepsilon+\mathbb R_+\setminus\{0\})\,\mathcal Z_J$ (the result then follows by induction). Let $x\in \mathcal Z_I$ and let $J=\{i\in\{1,\ldots, p\},\,\mathfrak M(x)_i>0\}$. Then, for all $y\in \mathcal Z_I$, there exists $\delta >0$ such that $\delta x\leq y$, so that, using the fact that $\mathfrak M$ is increasing and homogeneous, $\mathfrak M(y)\geq \delta \mathfrak M(x)$ and hence $J\subset \{i\in\{1,\ldots, p\},\,\mathfrak M(y)_i>0\}$. Reciprocally, $\{i\in\{1,\ldots, p\},\,\mathfrak M(y)_i>0\}\subset J$, so that $\mathfrak M(y)\in \mathbb R_+\mathcal Z_J$. Since on $\mathbb S$, the norm of $\mathfrak M$ is bounded from below away from 0 and from above, we have $\mathfrak M(\mathcal Z_I)\subset [\varepsilon_1,\varepsilon_2]\,\mathcal Z_J$ for some $\varepsilon_1,\varepsilon_2>0$. This implies that $\mathfrak M^2$ is uniformly continuous on $\mathcal Z_I$. We conclude by iteration.
\end{proof}

Let us now proceed with the proof of Proposition~\ref{prop:lyap} in two steps:  first we show that, for all $z\in\mathbb R_+^p\setminus\{0\}$,
    \begin{align}
    \label{eq:step1}
    \limsup_{n\to+\infty} \frac{1}{|n|^a} \mathbb E\left[Q_a(Z_1)\mathbf{1}_{Z_1\neq 0}\mid Z_0=\lfloor nz\rfloor \right]= (\lambda^*)^a Q_a(z);
    \end{align} 
    second, we extend this to a uniform convergence in $z$ by using the regularity of $\mathcal P$.

    \medskip\noindent \textit{Step 1.} As $\mathcal P$ is positively homogeneous,
    we have for any $z\in \mathbb R_+^p\setminus\{0\}$,
    \begin{align}
    \mathbb E\left[Q_a(Z_1)\mathbf{1}_{Z_1\neq 0}\mid Z_0=\lfloor nz\rfloor \right]&=\mathbb E\left[\left(\mathcal P\circ \xi\left(\sum_{i=1}^p \sum_{k=1}^{\lfloor n z_i\rfloor} V_{i,\cdot}^{(k,1)} \right)\right)^a\right]\nonumber\\
    &=|n|^a \mathbb E\left[\left(\mathcal P\left(\frac{1}{n} \xi\left(\sum_{i=1}^p \sum_{k=1}^{\lfloor n z_i\rfloor} V_{i,\cdot}^{(k,1)} \right)\right)\right)^a\right].\nonumber
    \end{align} 
    According to Lemma~30 in~\cite{FritschVillemonaisEtAl2022}, we have
    \begin{align*}
    \frac{1}{n} \xi\left(\sum_{i=1}^p \sum_{k=1}^{\lfloor n z_i\rfloor} V_{i,\cdot}^{(k,1)} \right)\xrightarrow[n\to+\infty]{a.s} \mathfrak M(z).
    \end{align*}
    Let $I\subset\{1,\ldots,p\}$ be such that $\mathfrak M(z)\in (\mathbb R_+\setminus\{0\}) \mathcal Z_I$, then the above convergence entails that, almost surely,  for $n$ large enough, $\frac{1}{n} \xi\left(\sum_{i=1}^p \sum_{k=1}^{\lfloor n z_i\rfloor} V_{i,\cdot}^{(k,1)} \right)\in (\mathbb R_+\setminus\{0\})\,\mathcal Z_J$ for some $J\supset I$. Using Corollary~7 in~\cite{FritschVillemonaisEtAl2022}, we have, for all $n\geq 1$,
    \begin{align*}
        \mathfrak M(z)\geq \frac1n
        \mathbb E\left[Z_1\mid Z_0=\lfloor nz\rfloor \right]
        =\frac1n\mathbb 
        E\left[\xi\left(\sum_{i=1}^p \sum_{k=1}^{\lfloor n z_i\rfloor} V_{i,\cdot}^{(k,1)} \right)\right],
    \end{align*}
    so that we also have $J\subset I$ almost surely.
    
%
    
    Since $\mathcal P$ restricted to $(\mathbb R_+\setminus\{0\})\,\mathcal Z_I$ is continuous, and since $\mathfrak M\circ \mathcal P=\lambda^* \mathcal P$ by~\eqref{eq:defP}, we deduce that
    \begin{align*}
    \mathcal P\left(\frac{1}{n} \xi\left(\sum_{i=1}^p \sum_{k=1}^{\lfloor n z_i\rfloor} V_{i,\cdot}^{(k,1)} \right)\right)\xrightarrow[n\to+\infty]{a.s} \mathcal P\left(\mathfrak M(z)\right)=\lambda^*\mathcal P(z).
    \end{align*}
    In addition, since $\xi$ is sub-affine, since  $\frac1n\sum_{i=1}^p \sum_{k=1}^{\lfloor n z_i\rfloor} V_{i,\cdot}^{(k,1)}$ is bounded in $L^\momentexpo$ by Assumption~(M) and $\mathcal P$ being positively homogeneous, we deduce that $\mathcal P\left(\frac{1}{n} \xi\left(\sum_{i=1}^p \sum_{k=1}^{\lfloor n z_i\rfloor} V_{i,\cdot}^{(k,1)}\right)\right)$ is also bounded in $L^\momentexpo$ (where $\momentexpo>a$), and we deduce from the dominated convergence theorem that 
    \begin{align*}
    \mathbb E\left[\left(\mathcal P\left(\frac{1}{n} \xi\left(\sum_{i=1}^p \sum_{k=1}^{\lfloor n z_i\rfloor} V_{i,\cdot}^{(k,1)} \right)\right)\right)^a\right]\xrightarrow[n\to+\infty]{} (\lambda^*)^a\mathcal P(z)^a.
    \end{align*}
    We deduce that~\eqref{eq:step1} holds true. This concludes the first step of the proof.
    
\medskip\noindent\textit{Step 2.}   Fix a non-empty $I\subset \{1,\ldots,p\}$ and let $\varepsilon_1>0$ be a small number that will be fixed later. Let $z^1,\ldots,z^{m_1} \in (1+\varepsilon_1)\mathcal Z_I$ and $m_1=m(\varepsilon_1)\geq 1$ be such that, for all $z\in\mathcal Z_I$, there exists $k=k(z,\varepsilon_1)\in\{1,\ldots,m_1\}$ such that $z\leq z^{k}$ and $|z-z^{k}|\leq 2\varepsilon_1$. The existence of this finite family $z^1,\ldots,z^{m_1}$ easily follows from a compactness argument. Let $n_1=n_1(\varepsilon_1)$ be large enough so that, for all $n\geq n_1$, for all $k\in\{1,\ldots,m_1\}$,
    \begin{align*}
    \frac{1}{|n|^a} \mathbb E\left[Q_a(Z_1)\mathbf{1}_{Z_1\neq 0}\mid Z_0=\lfloor nz^k\rfloor\right]\leq (1+\varepsilon_1) (\lambda^*)^a \mathcal P(z^k)^a.
    \end{align*} 
    Then, since $Q_a$ is non-decreasing and since $Z_1$ (with $Z_0=\lfloor nz\rfloor$) is stochastically non-decreasing with respect to $z$, we deduce that, for all $n\geq n_1$ and all $z\in \mathcal Z_I$,
    \begin{align}
    \frac{1}{|n|^a} \mathbb E\left[Q_a(Z_1)\mathbf{1}_{Z_1\neq 0}\mid Z_0=\lfloor nz\rfloor\right]&\leq (1+\varepsilon_1) (\lambda^*)^a \mathcal P(z^{k(z,\varepsilon_1)})^a\nonumber\\
    &\leq  (1+\varepsilon_1)  (\lambda^*)^a \mathcal P(z)^a\,\left(\sup_{x,y\in\mathcal Z_I, |(1+\varepsilon_1)x-y|\leq 2\varepsilon_1}\frac{\mathcal P((1+\varepsilon_1)x)}{\mathcal P(y)}\right)^a.\label{eq:ineqfirst}
    \end{align}
    Since $\mathcal P$ is uniformly continuous on $\mathcal Z_I$ and lower bounded away from $0$ on $\mathcal Z_I$, we deduce that there exists $\varepsilon_1$ small enough so that
    \[
    \theta_a:= (1+\varepsilon_1)^{a+1} (\lambda^*)^a \,\left(\sup_{x,y\in\mathcal Z_I, |(1+\varepsilon_1)x-y|\leq 2\varepsilon_1}\frac{\mathcal P(x)}{\mathcal P(y)}\right)^a < \theta_0.
    \]
    This and~\eqref{eq:ineqfirst} implies that
    \begin{align*}
    \frac{1}{|n|^a} \mathbb E\left[Q_a(Z_1)\mathbf{1}_{Z_1\neq 0}\mid Z_0=\lfloor nz\rfloor\right]&\leq \theta_a \mathcal P(z)^a.
    \end{align*}
    Since any  $z\in \mathbb{N}^p\setminus\{0\} \cap \left(\mathbb{R}_+\setminus\{0\}\right)\mathcal Z_I$ can be written as $n \frac{z}{|z|}$ with $n=|z|$ and $\frac{z}{|z|}\in\mathcal Z_I$, we deduce that, for any such $z$ with $|z|\geq n_1$, 
    \begin{align*}
    \frac{1}{|z|^a} \mathbb E\left[Q_a(Z_1)\mathbf{1}_{Z_1\neq 0}\mid Z_0=z\right]\leq \theta_a \mathcal P(z/|z|)^a.
    \end{align*} 
    Finally, the homogeneity of $\mathcal P$ allows us to conclude that Proposition~\ref{prop:lyap} holds true for all $z\in \mathbb{N}^p\setminus\{0\} \cap \left(\mathbb{R}_+\setminus\{0\}\right)\mathcal Z_I$. Since there are only finitely many subsets $I$ of $\{1,\ldots,p\}$, Proposition~\ref{prop:lyap} holds true for all $z\in\mathbb N^p\setminus \{0\}$.

\section{Proof of Theorem~\ref{thm:finitely-many-QSDs}}
\label{sec:proof-of-thm:finitely-many-QSDs}

The proof of the first statement relies on Theorem~4.1~\cite{ChampagnatVillemonais2022} and in part on some arguments of the proof of  Theorem~5.1 therein. The main additional difficulty is that our reference exponential parameter $\theta_0$ is \textit{a priori} larger than the one considered in this reference. 
We overcome this difficulty in Lemma~\ref{lem:useful} below. Another difficulty is that a key part of the argument relies on the assumption~(E1-4) in~\cite{ChampagnatVillemonais2017}, which does not apply directly when the Lyapunov parameter~$\theta_a$ (in Proposition~\ref{prop:lyap}) is only assumed strictly smaller than $\theta_0$. 
To overcome this difficulty,  we prove in Proposition~\ref{prop:E-with-theta0} that this last condition actually entails the assumption~(E1-4) in~\cite{ChampagnatVillemonais2017}.


 Fix $a\in(1,\momentexpo)$ such that $(\lambda^*)^a<\theta_0$ and set $H=Q_a/\inf Q_a$. According to Proposition~\ref{prop:lyap} and since $H(z)$ tends to $\infty$ when $|z|\to+\infty$, there exists $\theta_a<\theta_0$ such that
\begin{align}
\label{eq:thetaa}
\limsup_{|z|\to+\infty} \frac{\mathbb E_z(H(Z_1)\mathbf 1_{Z_1\neq 0})}{H(z)}\leq \theta_a< \gamma:=(\theta_a+\theta_0)/2 <\theta_0.
\end{align}
In particular, there exist only finitely many communication classes $E_1,\ldots,E_{k_0}$, $k_0\geq 1$, for the process $(Z_n)_{n\in\mathbb N}$ such that 
\begin{align*}
\forall i\in\{1,\ldots,k_0\},\,\exists z\in E_i,\ \text{such that}\ \frac{\mathbb E_z(H(Z_1)\mathbf 1_{Z_1\neq 0})}{H(z)}> \gamma.
\end{align*}
Note that $k_0\neq 0$ (otherwise $\theta_0\leq \gamma$).
%
%
For all $i\in\{1,\ldots,k_0\}$, we define $H_i=H\1_{E_i}$ and denote by $Y^{(i)}$ the process with state space $E_i\cup\{0\}$ defined by
\begin{align*}
Y_n^{(i)}=\begin{cases}
Z_n&\text{if }Z_n\in E_i\\
0  &\text{otherwise}.
\end{cases}
\end{align*}
Note that $0$ is an absorbing state for the process $Y^{(i)}$ as $E_i$ is a communication class. 
For all $i\in\{1,\ldots,k_0\}$, we denote by $\theta_{0,i}$ the absorption parameter of $Y^{(i)}$:
\begin{align*}
\theta_{0,i}=\sup_{z\in E_i}\sup\{\theta>0,\ \liminf_{n\to+\infty} \theta^{-n}\mathbb P_z(Y^{(i)}_n\neq 0)>0\}
\end{align*}
and we set
\begin{align*}
\bar\theta=\max_{i\in \{1,\ldots,k_0\}}\theta_{0,i}.
\end{align*}

We define the set 
\[
E'_0:=\mathbb N^p\setminus \left(\{0\}\cup \bigcup_{i=1}^{k_0} E_i\right).
\]
By definition of $E_1,\ldots,E_{k_0}$, we have
\begin{align}
    \label{eq:inEk}
    \forall z\in E'_0,\ \frac{\mathbb E_z(H(Z_1)\mathbf 1_{Z_1\neq 0})}{H(z)}\leq  \gamma.
\end{align}


In what follows, we make use of the following crucial lemma.

\begin{lemma}
    \label{lem:useful}
    We have $\bar\theta=\theta_0$.
\end{lemma}

\begin{proof}[Proof of Lemma~\ref{lem:useful}]
    Our aim is to show that, for all $\rho>\bar\theta\vee\gamma$ and all $z\in \mathbb N^p$, we have
\begin{align}
    \label{eq:conv2}
    \liminf_{n\to+\infty} \rho^{-n}\mathbb P_z(Z_n\neq 0)=0.
\end{align}
This shows that $\bar\theta\vee\gamma\geq \theta_0$ and hence, as by definition $\gamma<\theta_0$, that $\bar\theta\geq \theta_0$. The converse inequality is trivial, concluding the proof of Lemma~\ref{lem:useful}.

    For all $i\neq j\in \{1,\ldots,k_0\}$, we have
\begin{align*}
\forall z\in E_j,\ \mathbb P_z(\exists n\geq 0, Z_n\in E_i)>0\ \text{ or }\forall z\in E_j,\ \mathbb P_z(\exists n\geq 0, Z_n\in E_i)=0.
\end{align*}
If the condition on the left is satisfied, we write $E_i\prec E_j$, otherwise we write $E_i\not\prec E_j$. This defines a partial order on $\{1,\ldots,k_0\}$.
    For all $i\in\{1,\ldots,k_0\}$, we denote by $c(i)$ the maximal length of sets of distinct indices $i_1=i,i_2,\ldots,i_{c(i)}$ such that $E_{i_{c(i)}}\prec \cdots \prec E_{i_2}\prec E_{i_1}$. We also define, for all $\ell\geq 1$, the set
    \begin{align*}
        E_0^{\ell}:=\{x\in E'_0,\ \exists i\in\{1,\ldots,k_0\}\text{ with }c(i)=\ell\text{ and }x\to E_i,\text{ and }x\not\to E_j\ \forall j\text{ such that }c(j)>\ell \},
    \end{align*}
    where $x\to E_i$ means that 
    $\mathbb P_x(\exists n\geq 0,\, Z_n \in E_i)>0$, while $x\not\to E_j$ means that 
    $\mathbb P_x(\exists n\geq 0,\, Z_n \in E_i)=0$. We also set
    \begin{align*}
        E_0^{0}:=E'_0\setminus \cup_{\ell=1}^{\max_i c(i)} E_0^\ell.
    \end{align*}

    We obtain~\eqref{eq:conv2} by proving that, for any fixed $\rho>\bar\theta\vee\gamma$, for all $z\in \mathbb N^p\setminus \{0\}$,
    \begin{align}
    \mathbb E_z(H(Z_n)1_{Z_n\neq 0})&\leq   C_{\rho} \rho^n H(z)
    \label{eq:lemma-step1}
    \end{align}
    where $C_\rho$ does not depend on $z$.

    Remarking that, by definition, $E_0^0\cup\{0\}$ is an absorbing set, then \eqref{eq:lemma-step1} is immediate for all $z\in E_0^0$ by~\eqref{eq:inEk}.
    In the following, we first obtain in Step~1 two useful inequalities (see~\eqref{eq:ineqinE0l} and~\eqref{eq:ineqinCi} below). Then, we prove~\eqref{eq:lemma-step1} for all $z\in E_i \cup E_0^{c(i)}$,  $i\in\{1,\ldots,k_0\}$, by induction on $c(i)$ in Step~2.

    \medskip\textit{Step~1.}
 We first observe that, for all $\ell\geq 1$ and all $z\in E_0^{\ell}$,  using~\eqref{eq:inEk},
    \begin{align}
        \label{eq:ineqinE0l}
        \mathbb E_z\left(H(Z_n)\1_{Z_n\in E_0^\ell} \right)\leq \gamma^n H(z)\leq \rho^n H(z).
    \end{align}

    We now prove that, for all $i\in \{1,\ldots,k_0\}$ and all $z\in E_i$,
    \begin{align}
        \label{eq:ineqinCi}
    \mathbb E_z(H(Z_n) 1_{Z_n\in E_i})\leq C\rho^n H(z).    
    \end{align}   
In order to do so, we consider the set
\begin{align*}
    K_i=\left\{z\in  E_i,\  \frac{\mathbb E_z(H(Z_1)\mathbf 1_{Z_1\neq 0})}{H(z)}>\gamma\right\}.
\end{align*}
    For all  $z\in E_i$, since $\rho> \bar\theta$ and by definition of $\bar\theta$,
\begin{align}
    \label{eq:conv1}
    \liminf_{n\to+\infty} \rho^{-n}\mathbb P_z(Y_n^{(i)}\neq 0)=0.
\end{align}
Since $K_i$ is finite according to~\eqref{eq:thetaa}, we deduce that there exists a constant $C>0$ such that
\begin{align}
    \label{eq:C-rho}
    \mathbb P_z(Z_n\in E_i)&= \mathbb P_z(Y_n^{(i)}\neq 0)\leq  C\rho^n,\ \forall z\in K_i.
\end{align}
In what follows, the constant $C$ may change from line to line.
Moreover, denoting $\tau_{K_i}=\min\{\ell\geq 0,\ Z_\ell\in K_i\}$, we deduce from the definition of $H$ and of $K_i$ that
\begin{align}
    \label{eq:psi1}
    \mathbb P_z(Z_n\in E_i\text{ and }n< \tau_{K_i})\leq \mathbb E_z(H(Z_n)1_{Z_n\in E_i\text{ and }n< \tau_{K_i}})\leq \gamma^n H(z),\ \forall z\in E_i.
\end{align}
Hence, for all $z\in E_i$ and all $n\geq 0$, we have, using the strong Markov property at time $\tau_{K_i}$, 
\begin{align}
\nonumber
    \mathbb P_z(Z_n\in E_i)&\leq  \mathbb P_z(Z_n\in E_i\text{ and }n< \tau_{K_i})+\sum_{\ell=0}^n \mathbb P_z(\tau_{K_i}=\ell) \sup_{y\in K_i} \mathbb P_y(Z_{n-\ell}\in E_i)\\
\nonumber
    &\leq \gamma^n H(z)+\sum_{\ell=0}^n \mathbb P_z(\tau_{K_i}=\ell) C\rho^{n-\ell}\\
\nonumber
    &\leq \gamma^n H(z)+\sum_{\ell=0}^n \mathbb E_z(1_{\tau_{K_i}>\ell-1}H(Z_{\ell-1})) C\rho^{n-\ell}\\
\nonumber
    &\leq \gamma^n H(z)+ \sum_{\ell=0}^n \gamma^{\ell-1} H(z) C\rho^{n-\ell}\\
\label{eq.inter.prop.Ei}
    &\leq C \rho^n H(z),
\end{align}
where $C$ may depend on $\rho$ and $\gamma$.
Finally, from~\eqref{ineq:QSD}, we obtain for all $z\in E_i$
\[
    \mathbb E_z(H(Z_1)1_{Z_1\in E_i}) \leq \theta_a H(z) + \frac{C_a}{\inf Q_a}.
\]
Then, iterating this property and using~\eqref{eq.inter.prop.Ei}, we deduce that
\begin{align*}
    \mathbb E_z(H(Z_n)1_{Z_n\in E_i})&\leq \theta_a^n H(z)+\sum_{\ell=0}^{n-1} \frac{C_a}{\inf Q_a} \theta_a^{n-\ell-1} \mathbb P_z(Z_\ell\in E_i)\\
    &\leq \gamma^n H(z)+\sum_{\ell=0}^{n-1} \frac{C_a}{\inf Q_a} \gamma^{n-\ell-1}  C \rho^\ell H(z)\\
    &\leq C \rho^n H(z),
\end{align*}
where $C$ may depend on $\rho$ and $\gamma$.
This proves~\eqref{eq:ineqinCi} for all $z\in E_i$.

    \medskip\noindent\textit{Step 2.} Assume that there exists $c\geq 0$ such that~\eqref{eq:lemma-step1} holds true for all $z\in E_0^0,\ldots,E_0^c$ and all $z\in E_i$ with $c(i)\leq c$ (this is true for $c=0$ as we already saw that~\eqref{eq:lemma-step1} holds true for $z\in E_0^0$).
    If $c=\max_{i\in \{1,\ldots,k_0\}} c(i)$, then there is nothing to prove. Otherwise, let $i\in \{1,\ldots,k_0\}$ such that $c(i)=c+1$. 
    
    Denote by $T_i$ the first exit time from $E_i$ by $(Z_n)_{n\in\mathbb N}$
    and let us use a similar computation as in the end of Step~1. For all $z\in E_i$ and all $n\geq 0$, we have
    \begin{align*}
    \mathbb E_z(H(Z_n) 1_{Z_n\neq 0})&= \mathbb E_z(H(Z_n) 1_{Z_n\in E_i}) +\sum_{\ell=1}^n \mathbb E_z\left(1_{T_i=\ell} 1_{Z_\ell \neq 0}\mathbb E_{Z_\ell}\left(H(Z_{n-\ell})1_{Z_{n-\ell}\neq 0}\right)\right)\\
    &\leq C\rho^n H(z)+\sum_{\ell=1}^n \mathbb E_z\left(1_{T_i=\ell}1_{Z_\ell \neq 0} C_\rho \rho^{n-\ell} H(Z_{\ell})\right)
    \end{align*}
    where we used~\eqref{eq:ineqinCi}, the Markov property, the fact that, conditioning on $Z_0\in E_i$,
    \begin{align*}
    Z_{T_i}\in \{0\}\cup \bigcup_{\ell\leq c} E_0^\ell \cup \bigcup_{c(j)\leq c} E_j\text{ almost surely },
    \end{align*}
    and the induction assumption. Then, there exists $C_a>0$ such that
    \begin{align*}
     \mathbb E_z(H(Z_n) 1_{Z_n\neq 0})
    &\leq C\rho^n H(z)+\sum_{\ell=1}^n \mathbb E_z\left(1_{T_i>\ell-1} C_\rho \rho^{n-\ell} (\theta_a H(Z_{\ell-1})+C_a)\right)\\
    &\leq C\rho^n H(z)+\sum_{\ell=1}^n \mathbb E_z\left(1_{T_i>\ell-1} C_\rho \rho^{n-\ell} (\theta_a+C_a) H(Z_{\ell-1})\right)
    \end{align*}
    where we used the Markov property at time $\ell-1$ and Proposition~\ref{prop:lyap}. We deduce from~\eqref{eq:ineqinCi} that
    \begin{align*}
    \mathbb P_z(Z_n\neq 0)
    &\leq C\rho^n H(z)+C_\rho  (\theta_a+C_a)  C \sum_{\ell=1}^{n}  \rho^{n-1} H(z)\\
    &\leq \tilde C(n+1)\rho^n H(z)
    \end{align*}
    for some $\tilde C>0$, which depends on $\rho$ but not on $z$ neither $n$. Since this is true for any $\rho>\bar\theta\vee \gamma$, this proves~\eqref{eq:lemma-step1} for all $z\in E_i$ with $c(i)=c+1$.

    By the same arguments, using~\eqref{eq:ineqinE0l} instead of~\eqref{eq:ineqinCi} and the fact that, starting from $E_0^{c+1}$, at the exit time $\tilde T_i$ of $E_0^{c+1}$, we have 
    \begin{align*}
    Z_{\tilde T_i}\in \{0\}\cup \bigcup_{\ell\leq c} E_0^\ell \cup \bigcup_{c(j)\leq c+1} E_j\text{ almost surely },
    \end{align*}
    one shows that this also holds true for all $z\in E_0^{c+1}$.
    
    This concludes the proof of Lemma~\ref{lem:useful}.
\end{proof}

From now on, we assume without loss of generality that the communication classes $E_1,\ldots,E_{k_0}$ are indexed in the following way : there exists $k\in\{1,\ldots,k_0\}$ such that, for all $i\in\{k+1,\ldots,k_0\}$, $\theta_{0,i}<\bar\theta$, and, for all $i\in\{1,\ldots,k\}$, $\theta_{0,i}=\bar\theta=\theta_0$. For any  $i\in\{k+1,\ldots,k_0\}$, we have, fixing any $\rho\in(\max_{j\in\{k+1,\ldots,k_0\}}\theta_{0,j},\bar\theta)$ and using the same method as in Step~1 from the proof of Lemma~\ref{lem:useful} (see~\eqref{eq:ineqinCi}),
\begin{align*} 
\mathbb E_z(H(Z_n)1_{Z_n \in E_i})\leq C_\rho\rho^n H(z),\ \forall z\in E_i,
\end{align*}
where $C_\rho$ is a positive constant which may depend on $\rho$. Setting
\[
E_0:=E'_0\cup\bigcup_{i=k+1}^{k_0} E_i,
\]
one easily checks, using the last inequality and~\eqref{eq:inEk}, that, for any $r\in (\gamma\vee \rho,\bar\theta)$, there exists $C_r>0$ such that
\begin{align*} 
    \mathbb E_z(H(Z_n)1_{n<T_{E_0}})\leq C_r r^n H(z),\ \forall z\in E_0
\end{align*}
with $T_{E_0}$ the first exit time of $E_0$.

Note also that, from Proposition~\ref{prop:lyap} since $\inf H=1$,
 there exists $C_W>0$ (possibly larger than $\theta_0$) such that, for all $z\in\mathbb{N}^p\setminus\{0\}$, $\mathbb{E}_z(H(Z_1)1_{Z_1\neq 0})\leq C_w H(z)$.

 In order to apply Theorem~4.1 in~\cite{ChampagnatVillemonais2022}, which implies the first part of Theorem~\ref{thm:finitely-many-QSDs}, it then remains to prove that, for all $i\in \{1,\ldots,k\}$ and all $z\in E_i$, we have
\begin{align}
\label{eq:expo-conv-Ei}
\left|\theta_0^{-n}\mathbb E_z\left(f(Z_n)\mathbf{1}_{Z_n\in E_i}\right)-h_i(z)\mu_i(f)\right|\leq C\alpha^n H_i(z)\,\|f\|_{H_i},\quad\forall n\geq 0,\ \forall f\in L^\infty(H_i),
\end{align}
for some $C>0,\alpha\in(0,1)$, some non-negative non-zero function $h_i$ and some probability measure $\mu_i$ on $E_i$. In order to do so, its is sufficient to check Assumptions~(E1-4) in Section~2 of~\cite{ChampagnatVillemonais2017}. These conditions are recalled in the Appendix and we will more precisely check Assumptions~(E1), (E2'), (E3) and (E4), which is sufficient to obtain (E1-4) according to Proposition~\ref{prop:E-with-theta0}.  Fixing $i\in \{1,\ldots,k\}$, we define the set
\begin{align*}
K_i:=\left\{z\in  E_i,\  \frac{\mathbb E_z(H_i(Z_1)\mathbf 1_{Z_1\neq 0})}{H_i(z)}>\gamma\right\},
\end{align*}
which is finite according to~\eqref{eq:thetaa}.

Since $E_i$ is a communication class and the process is assumed to be aperiodic, there exists $n_1^i\geq 0$ such that 
\[c_1^i:=\inf_{x,y\in K_i}\mathbb P_x\left(Z_{n_1^i}=y\right)>0.\]
Note that $K_i\neq \emptyset$, by definition of $E_i$. Fix $z_i\in K_i$. Then, for all $z\in K_i$,
\[\mathbb P_z\left(Z_{n_1^i}\in \cdot \right)\geq \mathbb P_z\left(Z_{n_1^i}=z_i\right)\delta_{z_i}~(\cdot )\geq c_1^i \nu_i(\cdot \cap K),\]
for $\nu_i = \delta_{z_i}$. This entails (E1). In addition, taking $\varphi_1=H_i$ and $\theta_1=\gamma<\theta_0=\theta_{0,i}$, the condition (E'2) is an immediate consequence of the definition of $K_i$ and the fact that $z\mapsto \mathbb E_z(H_i(Z_1)\mathbf 1_{Z_1\in E_i})$ is bounded on $K_i$ by Proposition~\ref{prop:lyap}. 

For all $x,y\in K_i$ and $n\geq 0$, we have, using the Markov property at time $n_1^i$ and the fact that $\mathbb P_x(Z_{n}\in E_i)$ is non-increasing in $n$ (as $\{0\}$ is absorbing for the process $Y^{(i)}$),
\begin{align*}
\mathbb P_x(Z_n\in E_i)\geq \mathbb P_x(Z_{n\vee n_1^i}\in E_i)\geq \P_x(Z_{n_1^i}=y)\,\mathbb P_y(Z_{n\vee n_1^i-n_1^i}\in E_i)\geq c_1^i\,\mathbb P_y(Z_{n}\in E_i).
\end{align*}
This proves (E3).

Finally (E4) holds true since we assumed that the process $Z$ is aperiodic and $E_i$ is a communication class.
This concludes the proof of~\eqref{eq:eta}. In addition, any quasi-stationary distribution $\nu\in\mathcal M(Q_a)$ with absorption parameter $\theta_0$ satisfies $\nu(\{x,\ j(x)=0\})=1$ (see Proposition~2.3 in~\cite{ChampagnatVillemonais2022}). Integrating~\eqref{eq:eta} with respect to $\nu$ implies that $\nu$ is a convex combination of $\nu_1,\ldots, \nu_\ell$. This concludes
the first part of Theorem~\ref{thm:finitely-many-QSDs}.

\medskip
To conclude the proof of Theorem~\ref{thm:finitely-many-QSDs},  it remains to prove that the set $E=\{z\in \mathbb N^p\setminus\{0\},\,\sum_{i\in I}\eta_i(z)=0\}$ is finite and that any quasi-stationary distribution with absorption parameter strictly smaller than $\theta_0$ is supported by $E$.

We first observe that Theorem~4.1 in~\cite{ChampagnatVillemonais2022} also implies that $j\equiv 0$ on the set $E$. Assume on the contrary that $E$ is not finite and fix $z_0\in\mathbb N^p\setminus \{0\}$ such that $\sum_{i=1}^\ell\eta_i(z_0)>0$ and $j(z_0)=0
$ (existence of such a point $z_0$ is guaranteed by Proposition~2.3 in~\cite{ChampagnatVillemonais2022}). We observe that
\begin{align*}
\liminf_{n\to+\infty} \theta_0^{-n}\mathbb P_{z_0}(Z_n\neq 0)>0.
\end{align*}

Then, by primitivity of $\mathfrak M$, there exists $n_0$ such that for all $x\in \mathbb N^p\setminus\{0\}$ we have $\mathfrak M^{n_0}(x)>0$. This implies that for $k>0$ large enough we have $\mathbb P_{kx} (Z_{n_0} \geq z_0)>0$. Indeed, if it is not the case, we can conclude using Theorem~6 in \cite{FritschVillemonaisEtAl2022} that 
\[\mathfrak M^{n_0}(x)=\lim_{k\to +\infty} \dfrac{\mathbb E(Z_{n_0} \mid Z_0=kx)}{k}\leq \lim_{k\to +\infty}\dfrac{z_0}{k}=0,\]
which is a contradiction. Then,  we have for any $z$ large enough $\mathbb P_z(Z_{n_0}\geq z_0) >0$.
 Hence, using the fact that $E$ is not finite and hence contains arbitrarily large points, there exists $z\in E$ and $z'_0\geq z_0$ such that $\mathbb P_z(Z_{n_0}= z'_0)>0$. By super-additivity of $\xi$, we have $\mathbb P_{z'_0}(Z_n\neq 0)\geq \mathbb P_{z_0}(Z_n\neq 0)$ and hence
 \begin{align*}
 \liminf_{n\to+\infty} \theta_0^{-n}\mathbb P_{z}(Z_{n_0+n}\neq 0)>0.
 \end{align*}
 As we noticed that $j\equiv 0$ on the set $E$, this is not compatible with the definition of $E$ and~\eqref{eq:eta}. We have thus proved by contradiction that $E$ is finite.

Let $\nu_{QS}$ be a quasi-stationary distribution for $Z$ in $\mathcal M(Q_a)$ with absorption parameter $\theta_{QS}<\theta_0$. Then~\eqref{eq:eta}  implies that $\sum_{i=1}^\ell \nu_{QS}(\eta_i)=0$ and hence the support of $\nu_{QS}$ is included in $E$. 
This concludes the proof of the penultimate assertion of Theorem~\ref{thm:finitely-many-QSDs}.

 To conclude, it remains to prove that there are no quasi-stationary distributions for $Z$ in $\mathcal
M(Q_a)$ with absorption parameter strictly larger than $\theta_0$. This is a direct consequence of~\eqref{eq:eta}. Indeed, since $j$ is bounded, integrating this inequality with respect to a quasi-stationary distribution $\nu_{QS}$ with absorption parameter $\theta_{QS}$, and taking $f\equiv 1$, shows that 
\begin{align*}
    \limsup_{n\to+\infty} \theta_0^{-n}n^{-\|j\|_\infty}\theta_{QS}^n 
    <+\infty,
\end{align*}
and hence $\theta_{QS}\leq\theta_0$.

\section{Proof of Theorem~\ref{thm:unique-QSD}}
\label{sec:lastone}

  Fix $a\in(1,\momentexpo)$ such that $(\lambda^*)^a<\theta_0$, and consider the Lyapunov type property with constants $\theta_a$ and $C_a$ from Proposition~\ref{prop:lyap}.
We make use of Section~2 in~\cite{ChampagnatVillemonais2017}.
In a first step, we check that Assumption~E therein (recalled in the Appendix below) holds true for the process under consideration. In a second step, we prove~\eqref{eq:thmQSD2},~\eqref{eq:thmQSD1} and~\eqref{eq:thmQSD3} and~\eqref{eq:thmQSD4} using the results of~\cite{ChampagnatVillemonais2017}.

\medskip\noindent\textit{Step 1. Assumption~(E) holds true.}
Fix $\theta_1\in(\theta_a,\theta_0)$. As $\mathcal{P}$ is strictly positive on $\mathbb R_+^p\setminus\{0\}$ and positively homogeneous, we can fix $r_1\geq p$ large enough so that $Q_a(z)\geq \frac{C_a}{\theta_1-\theta_a}$ for all $|z|\geq r_1$, so that
\begin{align}
\label{eq:step0proof}
\mathbb E_z(Q_a(Z_1))\leq \theta_a Q_a(z)+C_a\leq \theta_1 Q_a(z),\quad\forall z\in\mathbb N^p\text{ such that }|z|\geq r_1.
\end{align}
Let us set  $K=\{z\in\mathbb N^p\setminus\{0\},\ |z|\leq r_1\}$ and $ \varphi_1=Q_a/\inf Q_a$ (note that $\inf Q_a>0$
 by Assumption~(P)). 

We deduce from the irreducibility and aperiodicity properties that there exists $n_1\geq 0$ such that 
\begin{align}
\label{eq:E1}
c_1:=\inf_{x,y\in K}\,\P_x(Z_{n_1}=y)>0.
\end{align}
Setting $\nu=\delta_{(1,\ldots,1)}$, this entails property (E1). 

We have $\inf  \varphi_1= 1$, so that the first line of (E2) is satisfied. Moreover, according to~\eqref{eq:step0proof}, the third line of (E2) also holds true. 

Choosing $\theta_2\in(\theta_1,\theta_0)$, we deduce from the definition of $\theta_0$ and the irreducibility of $Z$ that 
\begin{align*}
\theta_2^{-n}\min_{i\in \{1,\ldots,p\}}\mathbb P_{e_i}(Z_n\neq 0)\xrightarrow[n\to+\infty]{}+\infty,
\end{align*}
where  $e_i$ is the $i^{th}$ element of the canonical basis of $\mathbb N^p$.
Since, by the super-additivity of $\xi$, the process $Z$ is stochastically non-decreasing in the initial condition and since,  for all $z\in \mathbb N^p\setminus\{0\}$, there exists $i\in\{1,\ldots,p\}$ such that $z\geq e_i$, we deduce that
\begin{align*}
\theta_2^{-n}\inf_{z\in\mathbb N^p\setminus\{0\}}\mathbb P_z(Z_n\neq 0)\geq \theta_2^{-n}\min_{i\in \{1,\ldots,p\}}\mathbb P_{e_i}(Z_n\neq 0)\xrightarrow[n\to+\infty]{}+\infty.
\end{align*}
Let $n\geq 1$ be large enough so that $\theta_2^{-n}\inf_{z\in\mathbb N^p\setminus\{0\}}\mathbb P_z(Z_n\neq 0)\geq 1$, and set, for all $z\in \mathbb N^p\setminus\{0\}$
\begin{align*}
\varphi_2(z)=\frac1{\sum_{k=0}^{n-1} \theta_2^{-k}}\sum_{k=0}^{n-1} \theta_2^{-k} \mathbb P_z(Z_k\neq 0).
\end{align*}
We have, for all $z\in \mathbb N^p\setminus\{0\}$ and using the Markov property at time $1$,
\begin{align*}
\left(\sum_{k=0}^{n-1} \theta_2^{-k}\right)\mathbb E_z(\varphi_2(Z_1)\1_{Z_1\neq 0})&=\sum_{k=0}^{n-1} \theta_2^{-k} \mathbb P_z(Z_{k+1}\neq 0)\\
                                         &=\theta_2\sum_{k=1}^{n} \theta_2^{-k} \mathbb P_z(Z_{k}\neq 0)\\
                                         &=\theta_2\left(\sum_{k=0}^{n-1} \theta_2^{-k}\right)\varphi_2(z)+\theta_2\left(\theta_2^{-n} \mathbb P_z(Z_{n}\neq 0)-1\right)\\
                                         &\geq \theta_2 \left(\sum_{k=0}^{n-1} \theta_2^{-k}\right) \varphi_2(z).
\end{align*}
As a consequence the second and fourth lines of (E2) are satisfied. Note also that we have $\inf_{\mathbb N^p\setminus\{0\}} \varphi_2\geq \frac1{\sum_{k=0}^{n-1} \theta_2^{-k}}$.

For all $x,y\in K$ and $n\geq 0$, we have, using the Markov property at time $n_1$ and the fact that $\mathbb P_x(Z_{n}\neq 0)$ is non-increasing in $n$ (as $\{0\}$ is absorbing),
\begin{align*}
\mathbb P_x(Z_n\neq 0)\geq \mathbb P_x(Z_{n\vee n_1}\neq 0)\geq \P_x(Z_{n_1}=y)\,\mathbb P_y(Z_{n\vee n_1-n_1}\neq 0)\geq c_1\,\mathbb P_y(Z_{n}\neq 0).
\end{align*}
This proves (E3).

Finally (E4) holds true since we assumed that the process $Z$ is aperiodic and irreducible. 

We have thus proved that Assumption~(E) holds true for $Z$ absorbed when it leaves $\mathbb N^p\setminus\{0\}$. This concludes the first step.

\medskip\noindent\textit{Step 2. Conclusion of the proof.} 
Theorem~2.1 in~\cite{ChampagnatVillemonais2017} implies that there exists a unique quasi-stationary distribution $\nu_{QS}$ in $\mathcal M(Q_a)$ and that~\eqref{eq:thmQSD3} holds true (using the fact that in our case, $\varphi_2$ is lower bounded away from $0$). Theorem~2.3 in this reference also implies that there exists an associated non-negative eigenfunction $\eta_{QS}\in L^\infty(Q_a)$ with eigenvalue $\theta_{QS}>0$, where $\theta_{QS}$ is the absorption parameter associated to $\nu_{QS}$, such that $\inf_K \eta_{QS}>0$ and
\begin{align}
\label{eq:prooftheta}
\theta_{QS}^{-n}\mathbb P_\cdot(Z_n\neq 0)\xrightarrow[n\to+\infty]{L^\infty(Q_a)} \eta_{QS}(\cdot).
\end{align}
Since the process is irreducible, for all $z\in\mathbb N^p\setminus\{0\}$, there exists $n\geq 0$ such that $\mathbb P_z(Z_n=(1,\ldots,1))>0$ and hence 
\begin{align*}
\theta_{QS}^n\eta_{QS}(z)=\mathbb E_z(\eta_{QS}(Z_n)\mathbf 1_{Z_n\neq 0})\geq \eta_{QS}((1,\ldots,1))\mathbb P_z(Z_n=(1,\ldots,1))>0,
\end{align*}
so that $\eta_{QS}$ is positive.   We thus proved~\eqref{eq:thmQSD2} (for the absorption parameter $\theta_{QS}$ which we prove below to be equal to $\theta_0$). Finally, since $\xi$ is super-additive, $\theta_{QS}^{-n}\mathbb P_z(Z_n\neq 0)$ and hence $\eta_{QS}$ increases with $z$, so that $\eta_{QS}$ is lower bounded away from $0$.
Corollary~2.7 in~\cite{ChampagnatVillemonais2017} shows that~\eqref{eq:thmQSD1} holds true, and, taking $f\equiv 1$, also implies that $\theta_{QS}=\theta_0$. 
Finally,~\eqref{eq:thmQSD4} is an immediate consequence of Corollary~2.11 in~\cite{ChampagnatVillemonais2017}. This concludes the proof of Theorem~\ref{thm:unique-QSD}.


\appendix

\section{Sufficient criterion for the exponential convergence to a quasi-sta\-tio\-nary distribution}

In Section 2 of~\cite{ChampagnatVillemonais2017}, the authors state that Assumption~(E)  is sufficient to prove exponential convergence toward a quasi-stationary distribution. Let us recall this assumption, with the notations and settings of the present paper.


In what follows, we consider a process $(Z_n)_{n\in\mathbb N}$ evolving in a measurable state space $E\cup\{\partial\}$, where $\partial\notin E$ is an absorbing point.

\medskip\noindent\textbf{Assumption (E).} There exists a positive integer $n_1$, positive real constants
$\theta_1,\theta_2,c_1,c_2,c_3$, two functions $\varphi_1,\varphi_2: E\rightarrow \R_+$ and a probability measure $\nu$ on a
measurable subset $K\subset E$ such that
\begin{itemize}
    \item[(E1)] \textit{(Local Dobrushin coefficient).} $\forall z\in K$, 
    \begin{align*}
    \mathbb{P}_z(Z_{n_1}\in\cdot)\geq c_1 \nu(\cdot\cap K).
    \end{align*}
    \item[(E2)] \textit{(Global Lyapunov criterion).} We have $\theta_1<\theta_2$ and
    \begin{align*}
    &\inf_{z\in E} \varphi_1(z)\geq 1,\ \sup_{z\in K}\varphi_1(z)<\infty \\
    &\inf_{z\in K} \varphi_2(z)>0,\ \sup_{z\in E}\varphi_2(z)\leq 1,\\
    &\mathbb E_z(\varphi_1(Z_1)\mathbf{1}_{Z_1\neq \partial})\leq \theta_1\varphi_1(z)+c_2\mathbf 1_K(z),\ \forall z\in E\\
    &\mathbb E_z(\varphi_2(Z_1)\mathbf{1}_{Z_1\neq \partial})\geq \theta_2\varphi_2(z),\ \forall z\in E.
    \end{align*}
    \item[(E3)] \textit{(Local Harnack inequality).} We have
    \begin{align*}
    \sup_{n\in \Z_+}\frac{\sup_{y\in K} \mathbb{P}_y(Z_n\neq \partial)}{\inf_{y\in K} \mathbb{P}_y(Z_n\neq \partial )}\leq c_3.
    \end{align*}
    \item[(E4)] \textit{(Aperiodicity).} For all $z\in K$, there exists $n_4(z)$ such that, for all $n\geq n_4(z)$,
    \begin{align*}
    \mathbb{P}_z(Z_n\in K)>0.
    \end{align*}
\end{itemize}

Our aim is to show that the condition involving the function $\varphi_2$ can be replaced by a condition involving $\theta_0$, where
\begin{align}
    \label{eq:utheta_def_general}
    \utheta = \sup_{z\in E}\sup\{\theta>0,\ \liminf_{n\to+\infty} \theta^{-n}\mathbb P_z(Z_n\neq \partial)>0\},
\end{align}
More precisely, we consider the following assumption.

\begin{itemize}
    \item[(E2')]  We have $\theta_1<\theta_0$ and
    \begin{align}
        &\inf_{z\in E} \varphi_1(z)\geq 1,\ \sup_{z\in K}\varphi_1(z)<\infty \nonumber\\
        &\mathbb E_z(\varphi_1(Z_1)\mathbf{1}_{Z_1\neq \partial})\leq \theta_1\varphi_1(z)+c_2\mathbf 1_K(z),\ \forall z\in E.\label{eq:E2prime}
    \end{align}
\end{itemize}

\begin{proposition}
    \label{prop:E-with-theta0}
    Assume that there exists a positive integer $n_1$, positive real constants
$\theta_1,c_1,c_2,c_3$, a function $\varphi_1: E\rightarrow \R_+$ and a probability measure $\nu$ on a
measurable subset $K\subset E$ such that (E1), (E'2), (E3) and (E4) hold true. Then, for any $\theta_2\in(\theta_1,\theta_0)$, there exists $\varphi_2$ such that Assumption~(E2) holds true.
\end{proposition}



\begin{proof}
    In a first step, we prove that, for any $\theta_2\in (\theta_1,\theta_0)$, there exists $n\geq 1$ such that $\inf_{z\in K} \theta_2^{-n}\mathbb P_z(Z_n\in K)\geq 1$. In a second step, we conclude by building a function $\varphi_2$ which satisfies Assumption~(E2).
    
    \medskip\noindent \textit{Step 1.} Assume on the contrary that there exists $\theta_2\in (\theta_1,\theta_0)$ such that, for all $n\geq 1$,  $\inf_{z\in K} \theta_2^{-n}\mathbb P_z(Z_n\in K)<1$. Then, using (E1), we deduce that, for all $n\geq n_1$,
    \begin{align*}
    1>\inf_{z\in K} \theta_2^{-n}\mathbb P_z(Z_n\in K)\geq \theta_2^{-n} c_1 \mathbb P_{\nu}(Z_{n-n_1}\in K)
    \end{align*}
    and hence that, for all $n\geq 0$,
    \begin{align*}
    \mathbb P_{\nu}(Z_{n}\in K)< \frac{\theta_2^{n_1}}{c_1}\theta_2^{n}.
    \end{align*}
    Applying~\eqref{eq:E2prime} iteratively, we deduce that, for all $z\in E$, for all $n\geq 1$,
    \begin{align*}
    \mathbb E_z(\varphi_1(Z_n)\1_{Z_n\neq \partial})&\leq \theta_1^n\varphi_1(z)+ c_2\sum_{k=1}^n \theta_1^{n-k} \mathbb P_z(Z_{k-1}\in K).
    \end{align*}
    Integrating with respect to $\nu$ and using the two previous inequalities, we get
    \begin{align}
    \label{eq:nuPnvarphi1}
    \mathbb E_{\nu}(\varphi_1(Z_n)\1_{Z_n\neq\partial})\leq \theta_1^n\nu(\varphi_1)+ c_2\sum_{k=1}^n \theta_1^{n-k} \mathbb P_{\nu}( Z_{k-1} \in K)\leq \theta_1^n\nu(\varphi_1)+ c_2\sum_{k=1}^n \theta_1^{n-k} \frac{\theta_2^{n_1-1}}{c_1} \theta_2^{k} \leq C' \theta_2^n,
    \end{align}
    for some constant $C'$ where we used the fact that $\nu(\varphi_1)\leq \sup_K\varphi_1<+\infty$ and $\theta_1 < \theta_2$. 
    
    For all $z\in E$ and all $n\geq 0$, we have, denoting $\tau_K=\min\{k\geq 0,\ Z_k\in K\}$ and using the strong Markov property at time $\tau_K$, 
    \begin{align*}
    \mathbb P_z(Z_n\neq\partial)&\leq  \mathbb E_z(\varphi_1(Z_n) 1_{Z_n\neq\partial}1_{n< \tau_K})+\sum_{k=0}^n \mathbb P_z(\tau_K=k) \sup_{y\in K} \mathbb P_y(Z_{n-k}\neq \partial)\\
    &\leq \theta_1^n\varphi_1(z)+\sum_{k=0}^n \mathbb P_z(\tau_K=k) c_3 \mathbb P_{\nu}(Z_{n-k}\neq\partial),
    \end{align*}
    where we used~\eqref{eq:E2prime} for the first term in the right hand side, and~(E3) for the second term. Since we also have, by~\eqref{eq:E2prime} for all $k\geq 2$ and trivially for $k=1$, 
    \[
    \mathbb P_z(\tau_K=k)\leq \mathbb P_z(Z_{k-1}\neq \partial,\ k-1<\tau_K) \leq \mathbb E_z(\varphi_1(Z_{k-1}) 1_{Z_{k-1}\neq\partial}1_{k-1<\tau_K})\leq \theta_1^{k-1} \varphi_1(z),
    \]
    and since $\mathbb P_{\nu}(Z_{n-k}\neq \partial)\leq  \mathbb E_{\nu}(\varphi_1(Z_{n-k})\1_{Z_{n-k}\neq \partial}) $,    
    we deduce from the previous inequality and from~\eqref{eq:nuPnvarphi1} that
    \begin{align*}
    \mathbb P_z(Z_n\neq \partial)&\leq \left(\theta_1^n+\sum_{k=0}^n \theta_1^{k-1} c_3 C'\theta_2^{n-k}\right)\varphi_1(z).
    \end{align*}
    In particular, for all $z\in E$, for all $\theta>\theta_2>\theta_1$, 
    \begin{align*}
    \theta^{-n}\mathbb P_z(Z_n\neq \partial)\xrightarrow[n\to+\infty]{} 0,
    \end{align*}
    then, by definition of $\theta_0$, we get $\theta_0\leq \theta_2$, which contradicts assumption $\theta_2\in(\theta_1,\theta_0)$.
    
    This concludes the proof of our first step: for any $\theta_2\in (\theta_1,\theta_0)$, there exists $n\geq 1$ such that $\inf_{z\in K} \theta_2^{-n}\mathbb P_z(Z_n\in K)\geq 1$.
    
    \medskip\noindent\textit{Step 2.} Fix $\theta_2\in(\theta_1,\theta_0)$ and let $n\geq 1$ such that $\inf_{z\in K} \theta_2^{-n}\mathbb P_z(Z_n\in K)\geq 1$, and define the function
    \begin{align*}
    \varphi_2:z\in E\mapsto C_{\theta_2}\sum_{k=0}^{n-1} \theta_2^{-k}\mathbb P_z(Z_k\in K)
    \end{align*}
    with $C_{\theta_2}:=\left(\frac{1-\theta_2^n}{1-\theta_2}\right)^{-1}$. Then, $\inf_{z\in K} \varphi_2(z)>0$ and $\sup_{z\in E}\varphi_2(z)\leq 1$. In addition,
    \begin{align*}
    \mathbb E_z(\varphi_2(Z_1)\1_{Z_1\neq \partial})
    &=C_{\theta_2}\sum_{k=0}^{n-1} \theta_2^{-k} \mathbb P_z(Z_{k+1} \in K)=C_{\theta_2}\theta_2\sum_{k=1}^{n} \theta_2^{-k} \mathbb P_z (Z_k\in K)\\
    &=\theta_2\varphi_2(z)+C_{\theta_2}\theta_2\left(\theta_2^{-n}\mathbb P_z (Z_n\in K)-\1_K(z) \right).
    \end{align*}
    Since $\theta_2^{-n}\mathbb P_z (Z_n\in K)-\1_K(z)\geq 0$ for all $z\in K$ (by definition of $n$) and for all $z\in E\setminus K$ (since for such $z$, we have $1_K(z)=0$), we deduce that 
    \begin{align*}
    \mathbb E_z(\varphi_2(Z_1)\1_{Z_1\neq \partial})\geq \theta_2 \varphi_2(z),\ \forall z\in E.
    \end{align*}
    This concludes the proof of Proposition~\ref{prop:E-with-theta0}.
\end{proof}

\bibliographystyle{abbrv}

\end{document}